\documentclass{amsart}

\usepackage{amscd, amssymb, amsmath, amsthm}
\usepackage{mathabx} 
\usepackage{enumerate, latexsym, mathrsfs}
\usepackage{xypic}
\usepackage[
		bookmarks=true, bookmarksopen=true,%
    bookmarksdepth=3,bookmarksopenlevel=2,%
    colorlinks=true,%
    linkcolor=blue,%
    citecolor=blue]{hyperref}
\newtheorem{theorem}{Theorem}[section]
\newtheorem{lemma}[theorem]{Lemma}
\newtheorem{proposition}[theorem]{Proposition}
\newtheorem{corollary}[theorem]{Corollary}
\newtheorem{conjecture}[theorem]{Conjecture}

\theoremstyle{definition}
\newtheorem{definition}[theorem]{Definition}

\newtheorem{example}[theorem]{Example}

\newtheorem{remark}[theorem]{Remark}

\numberwithin{equation}{section}

\newcommand{\Natural}{{\mathbb N}}
\newcommand{\Real}{{\mathbb R}}
\newcommand{\Rational}{{\mathbb Q}}
\newcommand{\Complex}{{\mathbb C}}
\newcommand{\Integral}{{\mathbb Z}}

\title[{Virtual Euler class one}]{
A criterion for virtual Euler class one}
     
\author[Yi Liu]{%
        Yi Liu} 
\address{%
        Beijing International Center for Mathematical Research, Peking University\\
				Beijing 100871, China P.R.} 
\email{%
    liuyi@bicmr.pku.edu.cn}

\thanks{Partially supported by NSFC Grant 11925101, 
and National Key R\&D Program of China 2020YFA0712800}

\date{%
 \today} 

\begin{document}

\begin{abstract}
	Let $M$ be an oriented closed hyperbolic $3$--manifold.
	Suppose that $w$ is a rational second cohomology class of $M$ with dual Thurston norm $1$.
	Upon the existence of certain nonvanishing Alexander polynomials,
	the author shows that the pullback of $w$ to some finite cover of $M$
	is the real Euler class of some transversely oriented taut foliation on that cover.
	As application,
	the author constructs examples with first Betti number either $2$ or $3$,
	and partial examples with any first Betti number at least $4$,
	supporting Yazdi's virtual Euler class one conjecture.
\end{abstract}

\maketitle

\section{Introduction}
After disproving Thurston's Euler class one conjecture,
Yazdi posed a virtual remedy \cite[Question 9.4]{Yazdi_ecoc}.
The new conjecture is called the \emph{virtual Euler class one conjecture}.
The revised statement is as follows.

\begin{conjecture}[Yazdi]\label{virtual_ecoc}
	Let $M$ be an oriented closed hyperbolic $3$--manifold.
	For any even (integral) lattice point $w\in H^2(M;\Real)$
	of dual Thurston norm $1$,
	there exists some finite cover $\tilde{M}$ of $M$,
	such that the pullback $\tilde{w}$ of $w$ to $\tilde{M}$
	is the real Euler class of some transversely oriented taut foliation on $\tilde{M}$.	
\end{conjecture}

In this paper, we provide a criterion for Conjecture \ref{virtual_ecoc}
in terms of nonvanishing Alexander polynomials.
The criterion allows us to 
virtually realize any rational point 
on certain closed faces on the boundary of the dual Thurston norm unit ball.

The proof of our criterion suggests that 
the even lattice point condition 
becomes less restrictive in the virtual context.
It appears more natural to extend the statement of Conjecture \ref{virtual_ecoc}
to all rational points in $H^2(M;\Real)$ of dual Thurston norm $1$.
Note that Conjecture \ref{virtual_ecoc} has no content for $M$ of first Betti number $0$.
The conjecture holds obviously true for $M$ of first Betti number $1$, 
by Gabai's Euler class one theorem for vertices (see Corollary \ref{ecoc_vertex}).
We provide evidence to the rational point version of Conjecture \ref{virtual_ecoc}
by producing (a large variety of) examples of first Betti number $2$ or $3$.
We also obtain partial examples for any first Betti number $\geq4$.

We state our criterion as follows.
Theorem \ref{main_Alexander} is the main content, 
and Corollary \ref{main_Alexander_corollary} is supplementary.

\begin{theorem}\label{main_Alexander}
Let $M$ be an oriented closed hyperbolic $3$--manifold.
Denote by $\mathcal{B}_{\mathrm{Th}}(M)\subset H_2(M;\Real)$  the Thurston norm unit ball of $M$
and by $\mathcal{B}_{\mathrm{Th}^*}(M)\subset H^2(M;\Real)$ the dual Thurston norm unit ball of $M$.
Suppose that $F\subset \partial\mathcal{B}_{\mathrm{Th}^*}(M)$
and $F^\vee\subset \partial\mathcal{B}_{\mathrm{Th}}(M)$ 
are a dual pair of closed faces.
Suppose that $\psi\in H^1(M;\Integral)$ is a primitive cohomology class,
such that the Poincar\'e dual of $\psi$ 
lies in the cone over the interior of $F^\vee$.

If the Alexander polynomial $\Delta_M^\psi(t)$ does not vanish,
then for any rational point $w$ in the interior of $F$,
there exists some finite cyclic cover $\tilde{M}$ of $M$ dual to $\psi$,
such that the pullback $\tilde{w}$ of $w$ to $\tilde{M}$
is the real Euler class of 
some transversely oriented taut foliation on $\tilde{M}$.
\end{theorem}

\begin{corollary}\label{main_Alexander_corollary}
In the case of Theorem \ref{main_Alexander},
for any rational point $w$ in the boundary of $F$,
there also exists some finite cyclic cover $\tilde{M}$ of $M$,
dual to some primitive cohomology class in $H^1(M;\Integral)$,
such that the pullback $\tilde{w}$ of $w$ to $\tilde{M}$
is the real Euler class of 
some transversely oriented taut foliation on $\tilde{M}$.
\end{corollary}

Corollary \ref{main_Alexander_corollary} follows easily from Theorem \ref{main_Alexander}
and the fact that nonvanishing of $\Delta^\psi_M(t)$ is a projectively open condition
on the primitive cohomology class $\psi$ (see Lemma \ref{semicontinuity_nonvanishing}).
To be precise, any closed subface $E$ of $F$
is dual to a closed face $E^\vee$ that contains $F^\vee$ as a subface.
Provided the nonvanishing of $\Delta^\psi_M(t)$,
there exists some primitive cohomology class $\phi$ in the cone over the interior of $E^\vee$,
such that $\Delta^\phi_M(t)$ does not vanish (Lemma \ref{semicontinuity_nonvanishing}).
Applying Theorem \ref{main_Alexander}, 
any rational point $w$ in the interior of $E$ 
is virtually realizable by a taut foliation, 
on some finite cyclic cover of $M$ dual to $\phi$.

Theorem \ref{main_Alexander} is proved in Section \ref{Sec-criterion}.

Our criterion suggests 
a potential approach to Conjecture \ref{virtual_ecoc}.
To elaborate, suppose that $w\in H^1(M;\Real)$ is a rational point on
$\partial \mathcal{B}_{\mathrm{Th}^*}(M)$, 
lying in the interior of a closed face $F$. 
In general,
it is possible that $\Delta^\psi_M(t)$ vanishes for every primitive $\psi$
of Poincar\'e dual lying in the cone over the interior of $F^\vee$.
In that case, our criterion does not apply directly.
However, there are many finite covers of $M$ with arbitrarily large first Betti number.
We should even be able to construct a finite cover $M'$ of $M$,
such that the pullback $w'$ of $w$ to $M'$ lies in the interior of a closed face $F'$,
and such that $F'$ has codimension as large as wanted in $H^2(M';\Real)$.
The codimension of $F'$ in $H^2(M;\Real)$ 
is equal to the dimension of $(F')^\vee$ in $H_2(M;\Real)$.
Therefore, passing to $M'$ allows us more room to pick a testing $\psi'$.
As a prospective step,
we expect to find some nonvanishing $\Delta^{\psi'}_{M'}(t)$ 
for some $M'$ and $\psi'$.
Then, applying our criterion,
there will be some finite cyclic cover $M''$ of $M'$ dual to $\psi'$,
such that the pullback $w''$ of $w'$ to $M''$ is realizable by a taut foliation.
This will prove Conjecture \ref{virtual_ecoc}.

Currently available techniques of virtual constructions
seem to be insufficient for accomplishing the above strategy.
However, I do not see why the prospective step would be impossible either. 

We establish the following results as actual applications of our criterion.
Theorem \ref{main_example_full} provides positive examples for Conjecture \ref{virtual_ecoc}
with first Betti number $2$ or $3$.
Theorem \ref{main_example_partial} provides positive partial examples
with first Betti number $\geq4$.

\begin{theorem}\label{main_example_full}
For $b=2$ and for $b=3$, 
there exists some oriented closed hyperbolic $3$--manifold $M$ with $b_1(M)=b$,
such that 
every rational point $w\in\partial\mathcal{B}_{\mathrm{Th}^*}(M)$ 
is virtually realizable by a taut foliation.
Indeed, the pullback of $w$ to some finite cyclic cover of $M$
dual to some primitive cohomology class in $H^1(M;\Integral)$
is the real Euler class of some transversely oriented taut foliation
on the covering manifold.
\end{theorem}

\begin{theorem}\label{main_example_partial}
For any integer $b\geq4$,
there exists some oriented closed hyperbolic $3$--manifold $M$ with $b_1(M)=b$,
and there exist some opposite pair of top-dimensional 
closed faces $\pm F\subset \partial\mathcal{B}_{\mathrm{Th}^*}(M)$,
such that every rational point $w\in\pm F$ 
is virtually realizable by a taut foliation,
and indeed,
by similar means as specified in Theorem \ref{main_example_full}.
\end{theorem}

Both Theorems \ref{main_example_full} and \ref{main_example_partial} are proved in Section \ref{Sec-examples}.

It is natural to ask whether it is possible to strengthen Theorem \ref{main_example_partial}
to a similar statement as Theorem \ref{main_example_full} by improving the current construction.
A few unsuccessful attempts motivate the following suspection,
which we call the \emph{vanishing Alexander polynomial conjecture}.

\begin{conjecture}\label{vanishing_AP}
Let $b\geq4$ be an integer. 
For any orientable connected closed $3$--manifold $M$ with $b_1(M)=b$,
there exists some primitive cohomology class $\psi\in H^1(M;\Integral)$,
such that the Alexander polynomial $\Delta^\psi_M(t)$ vanishes.
\end{conjecture}

There has been a long-standing open question 
as to whether there exists some integer $B\geq4$,
such that every orientable closed hyperbolic $3$--manifold $M$ with $b_1(M)\geq B$
has an infinite collection of abelian finite covers 
with (uniformly) unbounded $b_1$ \cite[Question A]{Cochran--Masters}.
The question was raised by Cochran and Masters
as a former attempt to attack the virtual infinite Betti number conjecture,
prior to Agol's proof of the latter \cite{Agol_VHC}.

A proof of Conjecture \ref{vanishing_AP}
will imply a positive answer 
to the question of Cochran and Masters with $B=4$.
This implication has already been noticed 
by Cochran and Masters \cite[Section 2]{Cochran--Masters}
(see also Remark \ref{cyclic_Betti_remark}).
On the other hand, 
any counter example to Conjecture \ref{vanishing_AP}
for some $b\geq4$ will imply 
an improvement of Theorem \ref{main_example_partial} 
to similar statement as Theorem \ref{main_example_full},
for the same $b$ (see Proposition \ref{vanishing_AP_negative}).

\subsection*{Ingredients}
None of the constructions in this paper are extremely complicated.
Therefore, we only indicate some initial ideas toward the final constructions.

Let $M$ be an oriented connected closed $3$--manifold.
Suppose that $\mathscr{F}_0$ is a transversely oriented taut foliation on $M$
with a compact leaf $S$. 
Note that $S$ is necessarily oriented, connected, closed, and nonseparating.
If we have another transversely oriented taut foliation $\mathscr{F}_1$,
such that $\mathscr{F}_1$ also contains $S$ as a compact leaf,
then we can produce a hybrid foliation on a double cover $M'$ of $M$ dual to $S$,
by cyclically gluing two foliated blocks of the forms 
$(M\setminus S,\mathscr{F}_0)$ and $(M\setminus S,\mathscr{F}_1)$.
There are two ways to arrange these foliated blocks in order,
resulting in two different transversely oriented taut foliations 
$\mathscr{F}'_{01}$ and $\mathscr{F}'_{10}$ on $M'$.
They differ by a deck transformation,
so the real Euler classes $e(\mathscr{F}'_{01})$ and $e(\mathscr{F}'_{10})$ in $H^2(M';\Real)$
also differ by a deck transformation.
Generally speaking, 
these two real Euler classes could be linearly independent
from the pullbacks of $e_0=e(\mathscr{F}_0)$ and $e_1=e(\mathscr{F}_1)$.
In that case, we would have little to say.
However, 
if $b_1(M')$ is equal to $b_1(M)$, 
the deck transformation group acts trivially on $H^2(M';\Real)$,
so we must have $e(\mathscr{F}'_{01})=e(\mathscr{F}'_{10})$.
Moreover, we can figure out that they are both equal to $(e'_0+e'_1)/2$,
where $e'_0,e'_1\in H^2(M';\Real)$ denote the pullbacks of $e_0,e_1\in H^2(M;\Real)$.
This means that both $\mathscr{F}'_{01}$ and $\mathscr{F}'_{10}$
realize the pullback of the convex combination $(e_0+e_1)/2$ to $M'$.

Similarly, if $M''$ is a finite cyclic cover of $M$ dual to $S$ of degree $m=m_0+m_1$,
and if $b_1(M'')=b_1(M)$,
we can realize the pullback of $(m_0e_0+m_1e_1)/m$ on $M''$ by cyclically stacking $m_0$ copies 
of the $(M\setminus S,\mathscr{F}_0)$--block and $m_1$ copies of the $(M\setminus S,\mathscr{F}_1)$--block,
in whatever order.
The similar idea applies to more general situations,
for example, if $\mathscr{F}_0$ on $M$ contains $n$ homologous compact leaves $S_1,\cdots,S_n$,
and if each $S_i$ is also a compact leaf of some other $\mathscr{F}_i$ on $M$.

The above idea leads to the \emph{medley construction}, as we introduce in Section \ref{Sec-medley}.
We prove Theorem \ref{main_Alexander} using the medley construction.
We prepare our foliated building blocks by invoking available constructions 
due to Gabai \cite{Gabai_taut}, and due to Gabai--Yazdi \cite{Gabai--Yazdi}.
The assumed existence of a nonvanishing Alexander polynomial $\Delta^\psi_M(t)$
ensures uniformly bounded $b_1$ for all finite cyclic covers of $M$ dual to $\psi$.
So, in this case,
the real Euler classes for the medley foliations will be under good control after $b_1$ stabilizes.
See Section \ref{Sec-criterion} for detail.

The construction for proving Theorems \ref{main_example_full} and \ref{main_example_partial}
adopts a different procedure, 
except for the last step 
where we apply Theorem \ref{main_Alexander} and Corollary \ref{main_Alexander_corollary}.
In \cite[Section 1]{Cochran--Masters}, 
Cochran and Masters constructed 
many orientable closed hyperbolic $3$--manifolds of first Betti number $2$ or $3$.
For each of those manifolds,
every finite abelian cover has the same first Betti number as that of the base manifold.
Our construction follows a procedure mimicking theirs.

We start with some oriented connected closed $3$--manifold $N$
with well understood Alexander polynomials.
If we do a surgery on $N$ with nonzero rational coefficient 
along a \emph{totally rationally null-homologous} knot in $N$,
we can make sure that the resulting manifold $M$ has $H^1(M;\Integral)$
canonically isomorphic to $H^1(N;\Integral)$,
and accordingly, 
the Alexander polynomials of $M$ are the same as those of $N$
(possibly up to a nonzero scalar).
Many such surgeries yield a hyperbolic $M$.
By constructing a seed manifold $N$ with deliberately designed Alexander polynomials,
we apply the above procedure 
to prove Theorems \ref{main_example_full} and \ref{main_example_partial}.
See Section \ref{Sec-examples} for detail.
 
\subsection*{Organization}
In Section \ref{Sec-taut}, we recall preliminary facts regarding taut foliations.
In Section \ref{Sec-AP}, we recall preliminary facts regarding Alexander polynomials.
In Section \ref{Sec-medley}, 
we introduce the medley construction, and study its first properties.
This section is preparation for Section \ref{Sec-criterion}. 
In Section \ref{Sec-criterion}, 
we prove Theorem \ref{main_Alexander}.
In Section \ref{Sec-totally_rnh},
we introduce totally rationally null-homologous knots, and discuss surgery along such knots.
This section is preparation for Section \ref{Sec-examples}. 
In Section \ref{Sec-examples},
we prove Theorems \ref{main_example_full} and \ref{main_example_partial},
together with a simple observation on Conjecture \ref{vanishing_AP}.

\section{Taut foliations}\label{Sec-taut}
In this preliminary section, we collect facts 
regarding taut foliations in $3$--manifold topology.
For general facts regarding foliations on $3$--manifolds,
we refer to Part 3 in the textbook of Candel and Colon \cite{Candel--Conlon_book_ii}.

Let $M$ be an oriented connected closed irreducible (smooth) $3$--manifold. 
A codimension--$1$ foliation on $M$ (with smooth leaves and continuous holonomy)
is said to \emph{taut} if every leaf intersects some smoothly immersed loop
which is transverse to all the leaves.
To clarify our terminology,
any leaf is considered to be 
an injectively immersed, connected surface in $M$.
For transversely oriented codimension--$1$ foliations in $M$, 
the leaves are considered to be compatibly oriented.

For any transversely oriented taut foliation $\mathscr{F}$ on $M$,
the Euler class of $\mathscr{F}$ refers to the Euler class 
$e(T\mathscr{F})\in H^2(M;\Integral)$
of the oriented tangent plane distribution of $\mathscr{F}$ over $M$.
In this case, 
the real reduction of the Euler class, 
called the \emph{real Euler class} of $\mathscr{F}$
and denoted as $e(\mathscr{F})\in H^2(M;\Real)$,
must be an even lattice point 
contained in the dual Thurston norm unit ball 
$\mathcal{B}_{\mathrm{Th}^*}(M)$ of $M$;
see \cite[Corollary 1]{Thurston_paper_norm} and \cite[Proposition 3.12]{Yazdi_ecoc}.
This fact lays down fundamental constraints
to any transversely oriented taut foliations on $M$.

We briefly recall the definitions of 
the \emph{Thurston norm} $\|\cdot\|_{\mathrm{Th}}$ on $H_2(M;\Real)$
and \emph{dual Thurston norm} $\|\cdot\|_{\mathrm{Th}^*}$ on $H^2(M;\Real)$
as follows.
For simplicity, we only present with the closed case,
although the similar definitions work for the bounded case
using $H_2(M,\partial M;\Real)$ and $H^2(M,\partial M;\Real)$ instead.

In general, the Thurston norm is a seminorm 
(that is, a generalized norm allowing to be zero on some linear subspace)
and the dual Thurston norm is a partical norm
(that is, a generalized norm allowing to be infinite outside some linear subspace).
These are genuine norms when $M$ supports a hyperbolic structure.

To define $\|\cdot\|_{\mathrm{Th}}$,
first consider any integral class $\Sigma\in H_2(M;\Integral)$.
The Thurston norm $\|\Sigma\|_{\mathrm{Th}}$
is defined to be the minimum of the quantity
$\sum_{i=1}^{k}\,\max(-\chi(S_i),0)$,
where $S=S_1\sqcup\cdots\sqcup S_k$ ranges over 
all embedded closed oriented surfaces in $M$ representing $\Sigma$,
and where the connected components of $S$ listed as $S_1,\cdots,S_k$.
The function $\|\cdot\|_{\mathrm{Th}}$ on $H_2(M;\Integral)$
extends linearly over $H_2(M;\Rational)$
and continuously over $H_2(M;\Real)$,
yielding a unique and well-defined seminorm on $H_2(M;\Real)$
\cite[Theorem 1]{Thurston_paper_norm}.
For any $w\in H^2(M;\Real)$,
the value of $\|w\|_{\mathrm{Th}^*}$
is defined to be the supreme of
$|\langle w,\Sigma\rangle|$,
ranging over all $\Sigma\in H_2(M;\Real)$ with $\|\Sigma\|_{\mathrm{Th}}\leq1$,
which yields a unique and well-defined partial norm on $H^2(M;\Real)$.

Denote by $\mathcal{B}_{\mathrm{Th}}(M)\subset H_2(M;\Real)$ 
the unit ball of the Thurston norm $\|\cdot\|_{\mathrm{Th}}$
and by $\mathcal{B}_{\mathrm{Th}^*}(M)\subset H^2(M;\Real)$
the unit ball of the dual Thurston norm $\|\cdot\|_{\mathrm{Th}^*}$.
Thurston shows that there are finitely many even lattice points
$u_1,\cdots,u_s\in H_2(M;\Real)$,
such that $\|\Sigma\|_{\mathrm{Th}}$
is equal to the maximum among all $|\langle u_i,\Sigma\rangle|$,
for any $\Sigma\in H_2(M;\Real)$ \cite[Theorem 2]{Thurston_paper_norm}.
This characterize $\mathcal{B}_{\mathrm{Th}}(M)$ as the intersection
of the half-spaces defined by the affine linear inequalities
$\pm u_1\leq 1,\cdots,\pm u_s\leq 1$ on $H_2(M;\Real)$,
and characterizes $\mathcal{B}_{\mathrm{Th}^*}(M)$ as the affine convex hull
of the even lattice points $\pm u_1,\cdots,\pm u_s$ in $H^2(M;\Real)$.

Compact leaves play an especially important role
in studying transversely oriented taut foliations $\mathscr{F}$ on $M$
with real Euler class 
$e(\mathscr{F})\in\partial\mathcal{B}_{\mathrm{Th}^*}(M)$.
We collect the following relevant facts.

\begin{theorem}[Thurston]\label{norm_minimizing}
	Let $M$ be an oriented connected closed irreducible $3$--manifold.
	If $\mathscr{F}$ is a transversely oriented taut foliation on $M$,
	and if $\mathscr{F}$ contains the connected components
	of an embedded, oriented, closed surface $S$ in $M$ as compact leaves,
	then 
	$$\|[S]\|_{\mathrm{Th}}=\langle e(\mathscr{F}),[S]\rangle=-\chi(S),$$
	and hence,
	$$\|e(\mathscr{F})\|_{\mathrm{Th}^*}=1.$$
\end{theorem}

This is called the \emph{norm-minimizing property}
for compact leaves of transversely oriented taut foliations;
see \cite[Corollary 2]{Thurston_paper_norm}.

\begin{theorem}[Gabai]\label{converse_norm_minimizing}
	Let $M$ be an oriented connected closed irreducible $3$--manifold.
	If $S$ is an embedded, oriented, closed surface in $M$
	without sphere components, 
	and if the equality $\|[S]\|_{\mathrm{Th}}=-\chi(S)$ holds,
	then there exists a transversely oriented taut foliation $\mathscr{F}$ on $M$,
	such that $\mathscr{F}$ contains the connected components of $S$ as compact leaves.
\end{theorem}

This is a consequence of Gabai's sutured hierarchy theory; 
see \cite[Theorem 5.5]{Gabai_taut}.

As an application, consider any vertex $v$ on $\partial\mathcal{B}_{\mathrm{Th}^*}(M)$.
Dual to $v$ is some top-codimensional closed face $F^\vee$ on $\partial\mathcal{B}_{\mathrm{Th}}(M)$.
Take some embedded, oriented, closed, norm-minimizing surface $S$ representing a homology class 
$[S]\in H_2(M;\Integral)$ in the cone over $\mathrm{int}(F^\vee)$,
then there is some transversely oriented taut foliation $\mathscr{F}$ on $M$
containing the connected components of $S$ as compact leaves (Theorem \ref{converse_norm_minimizing}),
and $e(\mathscr{F})$ has to be $v$ in this case (Theorem \ref{norm_minimizing}).
Note that $\partial\mathcal{B}_{\mathrm{Th}^*}(M)\neq\emptyset$ if and only if $b_1(M)\geq1$.
Therefore, we obtain the following corollary; see also \cite[Theorem 1.4]{Gabai--Yazdi}.

\begin{corollary}[Gabai]\label{ecoc_vertex}
Let $M$ be an oriented connected closed irreducible $3$--manifold.
Suppose $b_1(M)\geq1$.
Then, every vertex on $\partial\mathcal{B}_{\mathrm{Th}^*}(M)$ is 
the real Euler class of some transversely oriented taut foliation on $M$.
\end{corollary}

\begin{theorem}[Gabai--Yazdi]\label{fully_marked}
	Let $M$ be an oriented closed hyperbolic $3$--manifold.
	Suppose that $\mathscr{F}$ is a transversely oriented taut foliation on $M$.
	Suppose that $S$ is an embedded oriented closed surface in $M$,
	such that the equality $\langle e(\mathscr{F}),[S]\rangle = -\chi(S)$ holds.
	Then, there exists some transversely oriented taut foliation $\mathscr{F}'$ on $M$
	homotopic to $\mathscr{F}$ as oriented plane distributions,
	and there exists some embedded oriented closed surface $S'$ in $M$
	homologous to $S$,
	such that $\mathscr{F}'$ contains the connected components of $S'$ as compact leaves.
\end{theorem}

This is called the \emph{fully marked surface theorem};
see \cite[Theorem 1.1]{Gabai--Yazdi}.

\section{Alexander polynomials}\label{Sec-AP}
In this preliminary section,
we review the multivariable polynomial and the Alexander polynomials with respect to primitive cohomology classes,
without involving extra twists.
For an overview of twisted Alexander polynomials in more generality,
we recommend the survey of Friedl and Vidussi \cite{Friedl--Vidussi_survey}.
We include direct proofs of 
two simple facts (Lemmas \ref{semicontinuity_nonvanishing} and \ref{cyclic_Betti})
for the reader's convenience.
We recall some special facts regarding $3$--manifolds.

For any (commutative) 
uniquely factorization domain (UFD) $R$,
and for any finitely generated module $V$ over $R$, 
we recall that the \emph{rank} of $V$ over $R$ refers to 
the dimension of $V\otimes_R\mathrm{Frac}(R)$
over the field of fractions $\mathrm{Frac}(R)$ of $R$.
If $R$ is Noetherian,
we recall that the \emph{order} of $V$ over $R$ refers to
any greatest common divisor of the zeroth elementary ideal $E_0(V)$,
which is unique up to a unit of $R$.
More precisely, as $R$ is Noetherian,
$V$ can be finitely presented as cokernel of some matrix $A$ over $R$ of size $m\times n$,
regarded as an $R$--linear homomorphism $R^n\to R^m$;
then $E_0(V)$ is the ideal of $R$ 
generated by all the $n\times n$--minors of $A$ (or the zero ideal if $m\leq n$),
which does not depend on the choice of $A$.
For any finitely generated module $V$ of a Noetherian UFD $R$,
the order of $V$ is nonzero if and only if the rank of $V$ is zero.

Let $X$ be a connected finite cell complex.
Denote by $H$ the free abelianization of $\pi_1(X)$.
Denote by $\tilde{X}^\#$ the covering space of $X$
corresponding to the kernel of $\pi_1(X)\to H$.
The deck transformation group action of $H$
endows $H_1(\tilde{X}^\#;\Integral)$ with a module structure
over the commutative group ring $\Integral[H]$.
As $X$ has only finitely many cells,
$H_1(\tilde{X}^\#;\Integral)$ is finitely generated over $\Integral[H]$.
	
For any primitive cohomology class $\psi\in H^1(X;\Integral)$,
denote by $\tilde{X}^\psi$ the infinite cyclic cover of $X$
dual to $\psi$.
Note that $H^1(X;\Integral)$ is a freely finitely generated over $\Integral$,
by the universal coefficient theorem.
A cohomology class in $H^1(X;\Integral)$ is said to be \emph{primitive}
if it is not divisible by any nonzero integer other than $\pm1$.
The universal free abelian covering projection $\tilde{X}^\#\to X$
naturally factors through $\tilde{X}^\psi$,
corresponding to the quotient homomorphism $\phi_*\colon H\to \Integral$.
Identifying the commutative group ring $\Integral[\Integral]$
with the Laurent polynomial ring $\Integral[t,t^{-1}]$
(in a fixed indeterminant $t$),
$H_1(\tilde{X}^\psi;\Integral)$ is finitely generated over $\Integral[t,t^{-1}]$.

\begin{definition}\label{AP_definition}
	Let $X$ be a connected finite cell complex.
	\begin{enumerate}
	\item
	Denote by $H$ the free abelianization of $\pi_1(X)$,
	and by $\tilde{X}^\#$ the universal free abelian cover of $X$.
	The (universal) \emph{multivariable Alexander polynomial} of $X$
	is defined to be the order of $H_1(\tilde{X}^\#;\Integral)$, 
	as a finitely generated module over $\Integral[H]$.
	It is denoted as a representative element $\Delta^\#_X$ in $\Integral[H]$,
	uniquely up to a unit of $\Integral[H]$.
	\item
	For any primitive cohomology class $\psi\in H^1(X;\Integral)$,
	denote by $\tilde{X}^\psi$ the infinite cyclic cover of $X$ dual to $\psi$.
	The \emph{Alexander polynomial} of $X$ with respect to $\psi$
	is defined to be the order of $H_1(\tilde{X}^\psi;\Integral)$,
	as a finitely generated module over $\Integral[t,t^{-1}]$.
	It is denoted as 
	a representative Laurent polynomial $\Delta^\psi_X(t)$ in $\Integral[t,t^{-1}]$,
	which is unique up to a monomial factor and a sign.
	Note that $\Integral[t,t^{-1}]^\times=\{\pm t^n\colon n\in\Integral\}$
	is the multiplicative group of units of $\Integral[t,t^{-1}]$.
	\end{enumerate}
\end{definition}

It is clear that $\Delta^\#_X$ and $\Delta^\psi_X(t)$ 
are essentially only invariants of the fundamental group $\pi_1(X)$.
In particular, 
it also makes sense to speak of $\Delta^\#_X$ and $\Delta^\psi_X(t)$ 
for any connected compact manifold $X$.

\begin{lemma}\label{semicontinuity_nonvanishing}
	Let $X$ be a connected finite cell complex.
	Suppose that $\psi\in H^1(X;\Integral)$ is a primitive cohomology class,
	such that $\Delta^\psi_X(t)$ does not vanish.
	Then, 
	there exists some open neighborhood $\mathcal{U}=\mathcal{U}(\psi)$
	of $\psi$ in $H^1(X;\Real)$, such that the following property holds.
	
	If $\phi\in H^1(X;\Integral)$ is a primitive cohomology class,
	and	if $\phi$ is a nonzero integral multiple of a rational cohomology class in $\mathcal{U}$, 
	then $\Delta^\phi_X(t)$ does not vanish.	
\end{lemma}

\begin{proof}
	For simplicity, we denote the cellular $\Integral$--chain complex of $\tilde{X}^\#$
	as $(C^\#_\bullet,\partial^\#_\bullet)$.
	For each dimension $n$,	
	$C^\#_n$ is freely finitely generated as a $\Integral[H]$--module.
	By choosing an ordering of the $n$--cells of $X$,
	and a lift of each $n$--cell to $\tilde{X}^\#$,
	the boundary operator $\partial^\#_n\colon C^\#_n\to C^\#_{n-1}$
	can be regarded as a matrix over $\Integral[H]$.
	
	For any $\alpha\in H^1(X;\Real)$,
	we obtain a $\Integral[\Real]$--chain complex 
	$(C^\alpha_\bullet,\partial^\alpha_\bullet)=
	(C^\#_\bullet\otimes_{\Integral[H]}\Integral[\Real],\partial^\#_\bullet\otimes1)$,
	where the $\Integral[H]$--module structure of $\Integral[\Real]$ is induced by
	the group ring homomorphism $\alpha_*\colon \Integral[H]\to\Integral[\Real]$.
	One may think of the elements in $\Integral[\Real]$ as 
	generalized Laurent polynomials	that are (finite) $\Integral$--linear combinations
	of real powers of the indeterminant $t$,
	and think of each $\partial^\alpha_n$ as obtained from $\partial^\#_n$
	by applying $\alpha_*$ to each matrix entry.
	Note that $\Integral[\Real]$ is a UFD and is flat over $\Integral[\Integral]$.
	
	For any primitive $\psi\in H^1(X;\Integral)$,
	the similarly constructed $\Integral[\Integral]$--chain complex
	$(C^\psi_\bullet,\partial^\psi_\bullet)$ can be identified with
	the cellular $\Integral$--chain complex of $\tilde{X}^\psi$.
	
	For any $\alpha\in H^1(X;\Real)$,
	we denote $H^\alpha_n=H_n(C^\alpha_\bullet,\partial^\alpha_\bullet)$.
	This is a $\Integral[\Real]$--module for each $n$.
	If $\alpha$ is nonzero, direct computation shows $H^\alpha_0\cong\Integral$.
	It follows that the $\Integral[\Real]$--rank of $H^\alpha_0$ is zero,
	so the $\Integral[\Real]$--rank of $\mathrm{Ker}(\partial^\alpha_1)$
	is equal to the $\Integral[\Real]$--rank of $\mathrm{Im}(\partial^\alpha_0)=C^\alpha_0$.
	This rank stays constant under nonzero perturbation of $\alpha$.
	On the other hand, the $\Integral[\Real]$--rank of $\mathrm{Im}(\partial^\alpha_2)$
	is the rank of the $\partial^\alpha_2$ as a matrix over $\Integral[\Real]$.
	This rank does not decrease under sufficiently small perturbations of $\alpha$,
	and it stays constant under nonzero $\Real$--scalar multiplication of $\alpha$.
	If $H^\alpha_1$ has $\Integral[\Real]$--rank zero,
	then the $\Integral[\Real]$--rank of $\mathrm{Im}(\partial^\alpha_2)$
	equal to the $\Integral[\Real]$--rank of $\mathrm{Ker}(\partial^\alpha_1)$,
	Therefore, if $H^\alpha_1$ has $\Integral[\Real]$--rank zero,
	and if $\alpha$ is nonzero, 
	then there exists some sufficiently small neighborhood $\mathcal{U}=\mathcal{U}(\alpha)$
	of $\alpha$ in $H^1(X;\Real)$, such that $H^{\beta}_1$ has $\Integral[\Real]$--rank zero
	for any $\beta\in H^1(X;\Real)$ 
	which is a nonzero real multiple of a real cohomology class in $\mathcal{U}$.
		
	Suppose that $\psi\in H^1(X;\Integral)$ is primitive,
	and suppose that $\Delta^\psi_X(t)$ does not vanish.
	Obtain an open neighborhood $\mathcal{U}=\mathcal{U}(\psi)$ of $\psi$ in $H^1(X;\Real)$ as above.
	For any primitive $\phi\in H^1(X;\Integral)$ which is an integral multiple of 
	some rational cohomology class in $\mathcal{U}$, 
	we infer that $H^\phi_1$ has $\Integral[\Real]$--rank zero.
	By flatness of $\Integral[\Real]$ over $\Integral[\Integral]$,
	we infer that $H_1(\tilde{X}^\psi;\Integral)$ has $\Integral[\Integral]$--rank zero.
	This is equivalent to the asserted property that $\Delta^\phi_X(t)$ does not vanish.
\end{proof}

\begin{lemma}\label{cyclic_Betti}
Let $X$ be a connected finite cell complex.
Suppose that $\psi\in H^1(X;\Integral)$ is a primitive cohomology class,
such that $\Delta_X^\psi(t)$ does not vanish.
Then, the following uniform bound holds for for all $m\in\Natural$,
$$b_1(X'_m)\leq \mathrm{deg}\left(\Delta^\psi_X\right)+1.$$
Here, $X'_m$ denotes the $m$--cyclic cover of $X$ dual to $\psi$;
$\mathrm{deg}(\Delta^\psi_X)$ denotes 
the floating degree of the Laurent polynomial $\Delta^\psi_X(t)$,
by which we mean 
the degree of the topmost nonzero term minus the degree of the bottommost nonzero term.
\end{lemma}

\begin{proof}
	For simplicity, we rewrite the infinite cyclic cover of $X$ dual to $\psi$
	as $X'_\infty=\tilde{X}^\psi$.
	If $\Delta_X^\psi(t)$ does not vanish,
	$H_1(X'_\infty;\Integral)$ has $\Integral[t,t^{-1}]$--rank zero.
	Since $\Rational$ is flat over $\Integral$,
	we obtain $H_1(X'_\infty;\Rational)\cong H_1(X'_\infty;\Integral)\otimes_\Integral\Rational$.
	Since $\Rational[t,t^{-1}]$ is a principal ideal domain (PID),
	it follows that $H_1(X'_\infty;\Rational)$ has finite dimension over $\Rational$,
	and dimension	is exactly the degree of the order $\Delta^\psi_X(t)$.
	The fundamental group $\pi_1(X)$ splits as a semidirect product
	$\pi_1(X'_\infty)\rtimes \langle z\rangle$.
	For any $m\in\Natural$, $\pi_1(X'_m)$ is the exactly the subgroup 
	$\pi_1(X'_\infty)\rtimes \langle z^m\rangle$.
	By abelianization, we obtain
	$$b_1(X'_m)\leq b_1(X'_\infty)+1=\mathrm{deg}\left(\Delta^\psi_X\right)+1,$$
	as asserted.
\end{proof}

\begin{remark}\label{cyclic_Betti_remark}
	Instead, if we suppose that $\Delta^\psi_X(t)$ vanishes,
	then $H_1(X'_\infty;\Integral)$ has nonzero rank over $\Integral[t,t^{-1}]$.
	In this case, the following well-known equality 
	implies that $b_1(X'_m)$ can be arbitrarily large.
	$$\mathrm{lim}_{m\to \infty}\,b_1(X'_m)/m
	=\mathrm{rank}_{\Integral[t,t^{-1}]}\left(H_1(X'_\infty;\Integral)\right).$$
	More precisely, the $\Integral[t,t^{-1}]$--rank of $H_1(X'_\infty;\Integral)$
	can be identified as the first $\ell^2$--Betti number of 
	the CW complex $X'_\infty$ with respect to the infinite cyclic 
	deck transformation group action \cite[Lemma 1.34]{Lueck_book};
	the latter is equal to the linear growth rate of $b_1(X'_m)$ in $m$,
	by the approximation conjecture (on $\ell^2$--Betti numbers) 
	which holds true for the infinite cyclic group
	\cite[Conjecture 13.1, Theorem 13.3, and Definition 13.9]{Lueck_book}.	
\end{remark}

For the multivariable Alexander polynomial $\Delta^\#_X$,
we can speak of its (floating) degree with respect to any $\alpha\in H^1(X;\Real)$.
To be precise, we express a representative $\Delta^\#_X$ as $a_1 h_1+\cdots a_s h_s$,
for some distinct $h_1,\cdots,h_s\in H$ and for some nonzero $a_1,\cdots,a_s\in \Integral$.
Then, for any $\alpha\in H^1(X;\Real)$, define
$$\mathrm{deg}_\alpha\left(\Delta^\#_X\right)=\max_{i,j} \left|\alpha(h_i)-\alpha(h_j)\right|,$$
(or zero if $\Delta^\#_X$ vanishes).
Note that the value of $\mathrm{deg}_\alpha$ 
does not change if we the representative changes by a unit of $\Integral[H]$.
In general, $\mathrm{deg}_\alpha(\Delta^\#_X)$ is greater than or equal to
the floating degree of $\alpha_*(\Delta^\#_X)$ (viewed as a generalized Laurent polynomial),
but not necessarily equal to because of possible canceling.
As a function in $\alpha$, 
the $\alpha$--degree of $\Delta^\#_X$ is a seminorm on $H^1(X;\Real)$,
which is also known as the \emph{Alexander norm} \cite{McMullen_norm}.

In the case of $3$--manifolds,
the Alexander norm supplies lower bounds to the Thurston norm,
and the multivariable Alexander polynomial determines 
the Alexander polynomials with respect to primitive cohomology classes.
We only state the case with oriented connected closed $3$--manifolds,
as this is the only case to be considered in the sequel.
The particular formulas below are only needed in Section \ref{Sec-examples}.

\begin{theorem}\label{MAP_properties}
	Let $M$ be an oriented connected closed $3$--manifold.
	\begin{enumerate}
	\item
	For any real cohomology class $\alpha\in H^1(M;\Real)$, the following inequality holds.
	$$\mathrm{deg}_\alpha\left(\Delta^\#_M\right)\leq \|\mathrm{PD}(\alpha)\|_{\mathrm{Th}}+
	\begin{cases}
	0 &b_1(M)\geq2\\
	2 &b_1(M)=1
	\end{cases}
	$$
	Here, $\mathrm{PD}$ denotes the Poincar\'e duality isomorphism.
	\item
	For any primitive cohomology class $\psi\in H^1(M;\Integral)$,
	the following equality holds in $\Integral[t,t^{-1}]$ holds up to a monomial factor and a sign.
	$$\Delta^\psi_M(t)\doteq \psi_*\left(\Delta^\#_M\right)\cdot
		\begin{cases}
	(t-1)^2 &b_1(M)\geq2\\
	1 &b_1(M)=1
	\end{cases}
	$$
	\end{enumerate}
\end{theorem}

See \cite[Theorem 1.1]{McMullen_norm} and \cite[Proposition 2.3]{Sun_vhsr} for these facts.

\begin{corollary}\label{MAP_properties_corollary}
	Let $M$ be an oriented connected closed $3$--manifold.
	For any primitive $\psi\in H^1(M;\Integral)$ with nonvanishing $\Delta^\psi_M(t)$,
	if $S$ is an oriented closed surface representing $\mathrm{PD}(\psi)$,
	and if $S$ contains no nonempty null-homologous subunion of connected components, 
	then $S$ is connected.
	In this case, the following inequality holds.
	$$\mathrm{deg}\left(\Delta^\psi_M\right) \leq 2\cdot\mathrm{genus}(S).$$
\end{corollary}

The assertion that $S$ is connected 
follows easily from Lemma \ref{cyclic_Betti}.
The asserted inequality is immediately implied by Theorem \ref{MAP_properties}.
In the special case when $\psi$ is a fibered class 
and when $S$ is a dual surface fiber,
$\Delta^\psi_M(t)$ can be computed as the characteristic polynomial
of the linear automorphism $f_*\colon H_1(S;\Integral)\to H_1(S;\Integral)$
induced by the monodromy $f\colon S\to S$ (well-defined up to isotopy).
This case achieves the equality in the asserted inequality.

\section{The medley construction}\label{Sec-medley}
In this section, we introduce a general construction
to cyclically stack a finite number of blocks 
with different foliation patterns.
The construction yields a foliated finite cyclic cover,
as precisely described in Lemma \ref{medley_construction}.
This is what we call the \emph{medley construction}.
We prove some basic properties of the medley construction
in Lemmas \ref{medley_construction_adjectives} and \ref{medley_construction_e}.

To avoid unnecessary restrictions, 
we do not assume the surfaces in Lemma \ref{medley_construction} to be connected.
We do not make use of such connectedness in subsequent sections either.
However, it may be noteworthy that
in the scenario of Theorem \ref{main_Alexander},
the surfaces to be used in the construction are indeed connected,
due to Corollary \ref{MAP_properties_corollary} (see Section \ref{Sec-criterion}).

\begin{lemma}\label{medley_construction}
Let $M$ be an oriented connected closed $3$--manifold.
Suppose that $\Sigma\in H_2(M;\Integral)$ is a primitive homology class.
Suppose $S_1,\cdots,S_n\subset M$ 
are embedded, mutually disjoint, oriented, closed subsurfaces,
such that each $S_i$ represents $\Sigma$, 
and such that each $M\setminus S_i$ is connected.
Suppose that 
$\mathscr{F}_0,\mathscr{F}_1,\cdots,\mathscr{F}_n$ are
transversely oriented foliations on $M$,
such that all components of $S_1\sqcup\cdots\sqcup S_n$
are leaves of $\mathscr{F}_0$,
and such that all components of each $S_i$ are leaves of $\mathscr{F}_i$.

Then, for any integers $m_0\geq1,m_1\geq0,\cdots,m_n\geq0$,
there exists 
some transversely oriented foliation $\mathscr{F}'$ 
on a cyclic cover $M'$ dual to $\Sigma$
of degree $m=m_0+m_1+\cdots+m_n$ over $M$
such that the following properties hold.
\begin{itemize}
\item
For each $S_i$, there are $m_0+m_i$ prescribed lifts of $S_i$ to $M'$,
whose components are all leaves of $\mathscr{F}'$.
\item
For each $X=M\setminus S_i$,
there are $m_i$ prescribed lifts of $(X,\mathscr{F}_i|_X)$ to $(M',\mathscr{F}')$.
The inward frontiers and the outward frontiers of these prescribed lifts
are all prescribed lifts of $S_i$ as above.
\item 
For every component $X$ of $M\setminus(S_1\sqcup\cdots\sqcup S_n)$,
there are $m_0$ prescribed lifts of $(X,\mathscr{F}_0|_X)$ to $(M',\mathscr{F}')$.
The inward frontiers and the outward frontiers of these prescribed lifts
are all among the prescribed lifts of $S_1,\cdots,S_n$ as above.
\item
All the above prescribed lifts are mutually disjoint, and 
they form $(M',\mathscr{F}')$ altogether.
\end{itemize}
\end{lemma}

\begin{proof}
	First we take an $m_0$--cyclic cover of $M$ dual to $\Sigma$
	with the pullback foliation of $\mathscr{F}_0$.
	Each $S_i$ and each component of $M\setminus(S_0\sqcup\cdots\sqcup S_n)$
	lifts to this cover, 
	since $\Sigma=[S_i]$ in $H_2(M;\Integral)$ for any $S_i$.
	The lifts of $S_0,\cdots,S_n$ 
	and the lifts of the components of $M\setminus(S_0\sqcup\cdots\sqcup S_n)$
	are declared as the prescribed lifts.
	At any prescribed lift of any $S_i$, 
	we can insert a copy of $(M\setminus S_i,\mathscr{F}_i|_{M\setminus S_i})$,
	(namely, cut along the prescribed lift of $S_i$
	and glue in between a copy of $M\setminus S_i$).
	This operation results in a new cyclic cover of $M$ dual to $\Sigma$
	of $1$ degree higher,
	introducing a new prescribed lift of $M\setminus S_i$,
	and a new prescribed lift of $S_i$.
	Redo this operation iterately for $m_i$ times for each $i=1,\cdots,n$.
	We end up with a cyclic cover $M'$ dual to $\Sigma$ 
	of degree $m=m_0+m_1+\cdots+m_n$ over $M$,
	and a transversely oriented foliation $\mathscr{F}'$ on $M'$,
	as asserted.
\end{proof}

\begin{lemma}\label{medley_construction_adjectives}
With the notations and assumptions in Lemma \ref{medley_construction},
if $\mathscr{F}_0,\mathscr{F}_1,\cdots,\mathscr{F}_n$ on $M$ are all taut,
then $\mathscr{F}'$ on $M'$ is also taut.
\end{lemma}

\begin{proof}
Suppose that $\mathscr{F}_0,\mathscr{F}_1,\cdots,\mathscr{F}_n$ on $M$ are all taut.
Fix a point in each component of $S_i$,
and denote their union as $Q_i\subset S_i$.
Since $\mathscr{F}_i$ is taut,
for each $q\in Q_i$, 
we can construct a path $\alpha_q$ from $q$ to some point in $Q_i$,
such that $\alpha_q$ is positively transverse to $\mathscr{F}_i$,
and $\alpha$ has no intersection with $S_i$ other than the endpoints.
Moreover, for any leaf $L\neq S_i$ in $\mathscr{F}_i$,
we may either construct a loop $\gamma_L$ in $M$ which has no intersection with $S_i$, 
or construct a path $\gamma_L$ 
which has both endpoints in $Q_i$ and no other intersection with $S_i$,
such that $\gamma$ is positively transverse to $\mathscr{F}_i$ and meets $L$.
Furthermore, we denote $S_0=S_1\sqcup\cdots\sqcup S_n$ and $Q_0=Q_1\sqcup\cdots \sqcup Q_n$,
then we obtain similar paths or loops $\alpha_q$ and $\gamma_L$ with respect to $\mathscr{F}_0$.
By concatenating lifts of the above $\alpha_q$ and $\gamma_L$,
we can form a loop $\delta'$ in $M'$ to meet any leaf $L'$ in $\mathscr{F}'$,
such that $\delta'$ is positively transverse to $\mathscr{F}'$.
Therefore, $\mathscr{F}'$ on $M'$ is taut.
\end{proof}

\begin{lemma}\label{medley_construction_e}
With the notations and assumptions in Lemma \ref{medley_construction},
the following Euler class formula holds in $H^2(M;\Integral)$:
$$\kappa_!(e(\mathscr{F}'))=m_0\cdot e(\mathscr{F}_0)+
m_1\cdot e(\mathscr{F}_1)+\cdots+m_n\cdot e(\mathscr{F}_n).$$
Here, $\kappa_!\colon H^2(M';\Integral)\to H^2(M;\Integral)$
denotes the umkehr homomorphism of the covering projection $\kappa\colon M'\to M$,
namely, $\kappa_!=\mathrm{PD}\circ \kappa_* \circ \mathrm{PD}$.
\end{lemma}

\begin{proof}
	By fixing a Riemannian metric on an oriented $3$--manifold $N$,
	any transversely oriented foliation $\mathscr{F}$ on $N$
	determines an oriented rank--$2$ real vector bundle $T\mathscr{F}$ over $N$,
	which can be identified as a complex line bundle $\mathcal{L}$ over $N$.
	In particular, $e(\mathscr{F})=c_1(\mathcal{L})$ holds in $H^2(N;\Integral)$.
	Moreover, if $\kappa\colon N'\to N$ is any finite covering projection,
	any complex line bundle $\mathcal{L}'$ over $N'$ 
	gives rise to a complex line bundle $\kappa_!\mathcal{L'}$ on $N$,
	such that the fiber of $\kappa_!\mathcal{L'}$ at $q\in N$ is 
	the complex tensor product of all	$\mathcal{L}'_{q'}$, 
	where $q'$ ranges over the lifts of $q$.
	
	When $N$ is closed, 
	the identity $c_1(\kappa_!\mathcal{L}')=\kappa_!(c_1(\mathcal{L}'))$ 
	holds in $H^2(N;\Integral)$.
	An easy way to see this is as follows.
	For any embedded loop $\alpha'\colon S^1\to N'$, 
	one can construct a map $N'\to \Complex\mathrm{P}^1$,
	such that every point outside a compact tubular neighborhood of $\alpha'$
	is mapped to $\infty$ in $\Complex\mathrm{P}^1\cong\Complex\cup\{\infty\}$,
	and in the tubular neighborhood, parametrized as $D^2\times S^1$,
	the map factors through the projection $D^2\times S^1\to D^2$,
	and descending to a map of $D^2$ to $\Complex\mathrm{P}^1\cong\Complex\cup\{\infty\}$
	which maps $\partial D$ to $\infty$ 
	and maps $\mathrm{int}(D^2)$ orientation-preserving homeomorphically onto $\Complex$.
	Composing with the standard inclusion of $\Complex\mathrm{P}^1$ into 
	$\Complex\mathrm{P}^\infty\simeq K(\Integral,2)$,
	the resulting map $N'\to \Complex\mathrm{P}^\infty$ represents an element in $H^2(N';\Integral)$,
	which is exactly the Poincar\'e dual of $[\alpha]\in H_1(N';\Integral)$,
	and which is also the pullback of the universal first Chern class.
	With this interpretation understood, we can rewrite 
	$c_1(\mathcal{L}')=\mathrm{PD}([\alpha'])$ for some $\alpha'$ as above,
	and reconstruct $\mathcal{L}'$ (isomorphically) 
	as a pullback of the universal complex line bundle
	via the above map $N'\to \Complex\mathrm{P}^\infty$.
	We may also assume that $D^2\times S^1$ 
	embeds into $N$ under $\kappa$,
	then our construction implies canonical isomorphisms
	$(\kappa_!\mathcal{L}')_q\cong \Complex$ for any $q\in N\setminus \kappa(D^2\times S^1)$,
	and $(\kappa_!\mathcal{L}')_q\cong \mathcal{L}'_{q'}$ 
	for any $q\in\kappa(D^2\times S^1)$, 
	where $q'$ is the unique inverse point in $D^2\times S^1$.
	So, $\kappa_!\mathcal{L}'$ over $N$ is canonically isomorphic to
	the pullback of the universal complex line bundle
	via	a classifying map $N\to \Complex\mathrm{P}^\infty$ 
	constructed the same way as above,
	with respect to the embedded loop $\kappa\circ\alpha'\colon S^1\to N$ 
	and the parametrized tubular neighborhood $\kappa(D^2\times S^1)$.
	This implies 
	$c_1(\kappa_!\mathcal{L}')=\mathrm{PD}(\kappa_*[\alpha'])$,
	and 
	the identity $c_1(\kappa_!\mathcal{L}')=\kappa_!(c_1(\mathcal{L}'))$ follows.
		
	To prove Lemma \ref{medley_construction_e}, 
	we obtain a complex line bundle $\mathcal{L}'$ over $M'$,
	associated to $\mathscr{F}'$,
	and similarly, 
	$\mathcal{L}_0,\mathcal{L}_1,\cdots,\mathcal{L}_n$ over $M$,
	associated to $\mathscr{F}_0,\cdots,\mathscr{F}_n$.
	By assumption, $\mathcal{L}_0$ and $\mathcal{L}_i$ are identical over $S_i$.
	
	To compute $\kappa_!(c_1(\mathcal{L}'))$,
	we compare $\mathcal{L}'$ 
	with another complex line bundle $\mathcal{L}'_0$ over $M'$,
	namely, the pullback of $\mathcal{L}_0$ to $M'$.
	Note that $\mathcal{L}'_0$ and $\mathcal{L}'$ are identical
	over any prescribed lift of any $S_i$,
	and over any prescribed lift of any component of $M\setminus(S_1\sqcup\cdots\sqcup S_n)$,
	so the complex line bundle 
	$\overline{\mathcal{L}'_0}\otimes \mathcal{L}'$
	is canonically isomorphic to the constant complex line bundle $\Complex$ over those prescribed lifts.
	Therefore, we obtain a canonical isomorphism
	of complex line bundles over $M$:
	$$\kappa_!\left(\overline{\mathcal{L}'_0}\otimes\mathcal{L}'\right)
	\cong
	\bigotimes_{i=1}^n \left(\overline{\mathcal{L}_0}\otimes \mathcal{L}_i\right)^{\otimes m_i}.
	$$
	This implies
	$$\kappa_!\left(c_1(\mathcal{L}')\right)-\kappa_!\left(c_1(\mathcal{L}'_0)\right)
	=\sum_{i=1}^n m_i\cdot \left(c_1(\mathcal{L}_i)-c_1(\mathcal{L}_0)\right)
	$$
	On the other hand, since $\mathcal{L}'_0$ is the pullback of $\mathcal{L}_0$ to $M'$,
	and since $M'$ is a cyclic cover of degree $m=m_0+m_1+\cdots+m_n$,
	we obtain
	$$\kappa_!\left(c_1(\mathcal{L}'_0)\right)=(m_0+m_1+\cdots+m_n)\cdot c_1(\mathcal{L}_0).$$
	Therefore, we obtain the identity in $H^2(M;\Integral)$:
	$$\kappa_!\left(c_1(\mathcal{L}')\right)
	=m_0\cdot c_1(\mathcal{L}_0)+m_1\cdot c_1(\mathcal{L}_1) +\cdots+ m_n\cdot c_1(\mathcal{L}_n),
	$$
	which is the same as the asserted formula in Lemma \ref{medley_construction_e}.	
\end{proof}

\section{The Alexander polynomial criterion}\label{Sec-criterion}

This section is devoted to the proof of Theorem \ref{main_Alexander}.

\begin{lemma}\label{medley_application}
	Let $M$ be an oriented closed hyperbolic $3$--manifold.
	Suppose that $F\subset \partial\mathcal{B}_{\mathrm{Th}^*}(M)$
	and $F^\vee\subset \partial\mathcal{B}_{\mathrm{Th}}(M)$ 
	are a dual pair of closed faces.
	Suppose that $w\in H^2(M;\Real)$ is a rational point in the interior of $F$.
	Suppose that $\Sigma\in H_2(M;\Integral)$ is a primitive homology class,
	such that $\Sigma$ lies in the cone over the interior of $F^\vee$.

	Then, there exist some finite cyclic cover $\tilde{M}$ of $M$ dual to $\Sigma$, 
	and some transversely oriented taut foliation $\tilde{\mathscr{F}}$ on $\tilde{M}$,
	such that the following equality holds in $H^2(M;\Real)$:
	$$\kappa_!(e(\tilde{\mathscr{F}}))=[\tilde{M}:M]\cdot w.$$
	Here, $\kappa_!\colon H^2(\tilde{M};\Real)\to H^2(M;\Real)$ denotes umkehr homomorphism for 
	the covering projection $\kappa\colon\tilde{M}\to M$.
\end{lemma}

\begin{proof}
	%
	%
	Enumerate the vertices of $F$ as
	$$v_1,\cdots,v_n\in F.$$
	By the Euler class one theorem for vertices (Corollary \ref{ecoc_vertex}),
	for each $v_i$, there exists some transversely orientable taut foliation $\mathscr{F}_i$ on $M$
	with real Euler class $e(\mathscr{F}_i)=v_i$.
	By assumption, $\langle e(\mathscr{F}_i),\Sigma\rangle=\|\Sigma\|_{\mathrm{Th}}$ 
	holds for each $\mathscr{F}_i$.
	By the fully marked surface theorem (Theorem \ref{fully_marked}),
	possiblying after modifying $\mathscr{F}_i$ 
	without altering the declared properties,
	we may assume that for each $\mathscr{F}_i$,
	there exists some embedded, oriented, closed subsurface $S_i$ of $M$ representing $\Sigma$,
	such that the components of $S_i$ are all leaves of $\mathscr{F}_i$.
		
	For any sufficiently large positive integer $d$,
	the complements $M\setminus S_1,\cdots,M\setminus S_n$ 
	can be lifted to a $d$--cyclic cover $M'$ of $M$ dual to $\Sigma$,
	such that the lifts have mutually disjoint closures in $M'$.
	Note that each $M\setminus S_i$ is connected, 
	because $\Sigma$ is primitive and because $\mathscr{F}_i$ are transversely oriented taut.
	Take some $d$ and $M'$ with this property, 
	and take some mutually disjoint lifts $S'_1,\cdots,S'_n$ for $S_1,\cdots,S_n$.
	We can also require $S'_i$ 
	to be the outward frontier of a lift of $M\setminus S_i$.
	This way, the lifts $S'_i$ are mutually homologous in $M'$.
	Denote by $\mathscr{F}'_1,\cdots,\mathscr{F}'_n$
	the pullback foliations of $\mathscr{F}_1,\cdots,\mathscr{F}_n$ to $M'$.
	Denote by $\Sigma'$ the common homology class $[S'_i]\in H_2(M';\Integral)$.
			
	By construction, $S'_1\sqcup\cdots\sqcup S'_n$ is
	an embedded, closed, oriented, norm-minimizing subsurface of $M'$,
	representing the nontorsion homology class $n\Sigma'\in H_2(M';\Integral)$.
	By Gabai's construction of taut foliations 
	from norm-minimizing surfaces (Theorem \ref{converse_norm_minimizing}),
	there exists some transversely oriented taut foliation $\mathscr{F}'_0$ on $M'$,
	such that the components of $S'_1\sqcup\cdots\sqcup S'_n$
	are all leaves of $\mathscr{F}'_0$.
	
	Observe that the image of the real Euler class $e(\mathscr{F}'_0)$
	under the umkehr homomorphism $H^2(M';\Integral)\to H^2(M;\Integral)$
	takes the form $d\cdot v_0$, 
	for some rational point
	$$v_0\in F$$ 
	depending on $(M',\mathscr{F}'_0)$.
	In fact,
	as $\Sigma$ lies in the cone over the interior of $F^\vee$,
	the linear function $u\mapsto \langle u,\Sigma\rangle$
	on $\mathcal{B}_{\mathrm{Th}}(M)$ maximizes exactly on $F$
	(with maximum value $\|\Sigma\|_{\mathrm{Th}}$).
	Hence, 
	the linear function $u'\mapsto \langle u',\Sigma'\rangle$
	on $\mathcal{B}_{\mathrm{Th}}(M)$ maximizes exactly on $F'$
	(with the same maximum value),
	where $F'$ denotes the maximal closed face of 
	$\partial \mathcal{B}_{\mathrm{Th}^*}(M')$
	that contains the pullback of $F$.
	By construction, we obtain
	$\langle e(\mathscr{F}'_0),\Sigma'\rangle=-\chi(S_i)=\|\Sigma'\|_{\mathrm{Th}}$.
	This shows that $e(\mathscr{F}'_0)$ lies in $F'$,
	so its image under the umkehr homomorphism lies in $d\cdot F$.
	
	With the vertices $v_1,\cdots,v_n$ of $F$ and the rational point $v_0\in F$ as above,
	we can find rational numbers $\mu_0>0,\mu_1\geq0,\cdots,\mu_n\geq0$ 
	with $\mu_0+\mu_1+\cdots+\mu_n=1$, such that 
	the given rational point $w\in \mathrm{int}(F)$	can be written as
	$$w=\mu_0\cdot v_0+\mu_1\cdot v_1+\cdots+\mu_n\cdot v_n.$$
	For example, we can choose some rational $\mu_0>0$ small enough, 
	making sure that $w-\mu_0\cdot v_0$ still lies in $\mathrm{int}(F)$,
	then $w-\mu_0\cdot v_0$ can be written as a convex combination 
	$\mu_1\cdot v_1+\cdots+\mu_n\cdot v_n$.
	Note that this is where we need the assumption $w\in\mathrm{int}(F)$
	(indeed, the construction would fail for any $w\in\partial F$ 
	if $v_0$ happens to lie in $\mathrm{int}(F)$).
	Take some $(\mu_0,\mu_1,\cdots,\mu_n)$ as above, 
	and take any sufficiently divisible positive integer $D$,
	such that $m'_0=\mu_0 D/d,m'_1=\mu_1 D/d,\cdots,m'_n=\mu_n D/d$ are all integers.
	
	Finally, we apply Lemma \ref{medley_construction} to the setting with 
	$M'$, $\Sigma'$, and $(\mathscr{F}'_0,\mathscr{F}'_1,\cdots,\mathscr{F}'_n)$,
	using the integers $(m'_0,m'_1,\cdots,m'_n)$.
	Then, we obtain a $D/d$--cyclic cover $\tilde{M}$ over $M'$ dual to $\Sigma'$,
	and a transversely oriented foliation $\tilde{\mathscr{F}}$ on $\tilde{M}$.
	By Lemma \ref{medley_construction_adjectives},
	$\tilde{\mathscr{F}}$ is taut, as $\mathscr{F}'_0,\mathscr{F}'_1,\cdots,\mathscr{F}'_n$ are all taut.
	Note that $\tilde{M}$ is also a $D$--cyclic cover of $M$ dual to $\Sigma$.
	By Lemma \ref{medley_construction_e}, we compute in $H^2(M;\Real)$:
	$$\kappa_!(e(\tilde{\mathscr{F}}))=D\cdot(\mu_0\cdot v_0+\mu_1\cdot v_1+\cdots+\mu_n\cdot v_n)=[\tilde{M}:M]\cdot w,$$
	as asserted.	
\end{proof}.

With the above preparation, we prove Theorem \ref{main_Alexander} as follows.

Let $M$ be an oriented closed hyperbolic $3$--manifold.
Let $F\subset \partial\mathcal{B}_{\mathrm{Th}^*}(M)$
and $F^\vee\subset \partial\mathcal{B}_{\mathrm{Th}}(M)$ 
be a dual pair of closed faces.
Let $w\in H^2(M;\Real)$ be a rational point in the interior of $F$.
Let $\psi\in H^1(M;\Integral)$ be a primitive cohomology class,
such that $\mathrm{PD}(\psi)$ lies in the cone over the interior of $F^\vee$.

If the Alexander polynomial $\Delta_M^\psi(t)$ does not vanish,
the first Betti number is uniformly bounded 
for all finite cyclic covers of $M$ dual to $\psi$ (Lemma \ref{cyclic_Betti}).
Suppose the maximum first Betti number is achieved by
a $d$--cyclic cover $M'$ of $M$ dual to $\psi$.
Note that the pullback $\psi'\in H^1(M';\Integral)$ has divisibility $d$,
so $\Sigma'=\mathrm{PD}(\psi'/d)$ is a primitive cohomology class.
Denote by $w'\in H^2(M';\Real)$ the pullback of $w$ to $M'$.
Denote by $F'\subset\partial \mathcal{B}_{\mathrm{Th}^*}(M')$
the minimal closed face containing the pullback of $F$.
We observe $w'$ lies in $\mathrm{int}(F')$.
and $\Sigma'$ lies in the cone over the interior of 
$(F')^\vee \subset\partial \mathcal{B}_{\mathrm{Th}}(M')$.
Applying Lemma \ref{medley_application} 
to $M'$ with respect to $w'$ and $\Sigma'$,
we obtain some finite cyclic cover $\tilde{M}$ of $M'$ dual to $\Sigma'$,
and some transversely oriented taut foliation $\tilde{F}$ on $\tilde{M}$,
such that the image of the real Euler class $e(\tilde{F})$
under the umkehr homomorphism $H^2(\tilde{M};\Real)\to H^2(M';\Real)$
is equal to $[\tilde{M}:M]\cdot w'$.

Since $M'$ already maximizes the first Betti number among all finite cyclic covers 
of $M$ dual to $\psi$,
the umkehr homomorphism $H^2(\tilde{M};\Real)\to H^2(M';\Real)$
is an isomorphism.
It follows that $e(\tilde{F})\in H^2(\tilde{M};\Real)$ 
is equal to the pullback $\tilde{w}$ of $w'\in H^2(M';\Real)$ to $\tilde{M}$,
which is also the pullback of $w\in H^2(M;\Real)$ to $\tilde{M}$.
Therefore,
we have verified that 
$\tilde{M}$ and $\tilde{\mathscr{F}}$ are as asserted in Theorem \ref{main_Alexander}.

This completes the proof of Theorem \ref{main_Alexander}.

\section{Totally rationally null-homologous knots}\label{Sec-totally_rnh}

For any oriented $3$--manifold $M$, 
an oriented null-homologous knot $K\subset M$
refers to a locally flatly embedded, oriented circle
which bounds some embedded, oriented compact subsurface of $M$.
In this case, there is a meridian-longitude pair with respect to $K$,
similar as usual for classical knots in $S^3$.
To be precise, take a compact tubular neighborhood of $K$ in $M$,
parametrized as $\partial D^2\times K$,
then the \emph{meridian} refers to the slope 
(that is, the isotopy class of an oriented essential simple closed curve)
on $\partial D^2\times K$ parallel to the first factor;
require that the parametrization makes the second factor 
null-homologous in $M\setminus K$,
then the \emph{longitude} refers to the slope parallel to the second factor.

For any rational number $p/q$,
expressed uniquely with coprime integers $p\geq0$ and $q\neq0$,
the $p/q$--\emph{surgery} of $M$ along $K$ follows
the similar standard convention as in classical knot theory.
Namely, the resulting manifold
is obtained from $M\setminus \mathrm{int}(D^2\times K)$
by Dehn filling with another solid torus,
such that the slope represented by $p$ times the meridian plus $q$ times the longitude
bounds a disk in the filling solid torus.
We denote the resulting manifold of the $p/q$--surgery of $M$ along $K$
as $M_{p/q}(K)$.

The core of the filling solid torus generates a $\Integral$--submodule 
of $H_1(M_{p/q}(K);\Integral)$, which is isomorphic to $\Integral/p\Integral$;
the quotient of $H_1(M_{p/q}(K);\Integral)$ by this submodule
is naturally isomorphic to $H_1(M;\Integral)$,
generalizing the similar fact for surgery on classical knots in $S^3$.

\begin{definition}\label{totally_rnh_def}
For any $3$--manifold $M$, 
a null-homologous knot $K\subset M$ is said to be
\emph{totally rationally null-homologous}
if there exists some compact neighborhood $R$ of $K$ in $M$,
such that $K$ is null-homologous in $R$,
and such that the inclusion map of $R$ into $M$
induces a zero homomorphism $H_1(R;\Rational)\to H_1(M;\Rational)$.
\end{definition}

If we require a stronger property that $R\to M$ is null-homotopic,
we will recover the definition of a \emph{totally null-homotopic} knot,
as appeared in \cite[Section 4]{Boileau--Wang_bundle}.
In particular, the following lemma is trivial implication
of \cite[Proposition 4.2]{Boileau--Wang_bundle}.
We include a sketch of the construction for the reader's convenience.

\begin{lemma}\label{hyperbolic_totally_rnh}
For any oriented connected closed $3$--manifold $M$,
there exists an oriented, null-homologous, 
and totally rationally null-homologous knot $K\subset M$,
such that $M\setminus K$ supports a complete hyperbolic structure of finite-volume.
Hence, for all but finitely many rational numbers $p/q$,
the surged manifold $M_{p/q}(K)$ supports a hyperbolic structure.
\end{lemma}

\begin{proof}
	Take a wedge of circles $\bigvee^r S^1$ with $r\geq2$.
	Take a map $\bigvee^r S^1\to M$ which induces a zero homomorphism 
	$H_1(\bigvee^r S^1;\Rational)\to H_1(M;\Rational)$.
	The map $\bigvee^r S^1\to M$ can always be homotoped to an embedding,
	such that a compact regular neighborhood $R$ of the embedded image
	is homeomorphic to a handlebody of genus $r$,
	and such that $M\setminus\mathrm{int}(R)$ is simple
	(that is, irreducible, boundary-irreducible, atoroidal, and acylindrical).
	Take an oriented null-homologous knot $K\subset R$ 
	with a compact tubular neighborhood $D^2\times K$,
	such that $R\setminus \mathrm{int}(D^2\times K)$ is simple.
	
	By construction, $K$ is totally rationally null-homologous in $M$.
	Moreover,
	$M\setminus K$ supports a complete hyperbolic structure of finite volume,
	by Thurston's hyperbolization theorem.
	Hence,
	for all but finitely many rational numbers $p/q$,
	the surged manifold $M_{p/q}(K)$ supports a hyperbolic structure,
	by Thurston's hyperbolic Dehn filling theorem.
	
	In our sketched construction above,
	the homotopic modifications for ensuring simplicity of the resulting manifolds
	can be done by (repeatedly) using tricks due to Myers.
	See \cite[Proof of Proposition 4.2]{Boileau--Wang_bundle} 
	for more detail.
\end{proof}

\begin{lemma}\label{surgery_Alexander}
Let $M$ be an oriented connected closed irreducible $3$--manifold.
Suppose that $K\subset M$ is 
an oriented, null-homologous, and totally rationally null-homologous knot.
For any coprime integers $p\geq0$ and $q\neq0$,
denote by $M_{p/q}(K)$ the $p/q$--surgery of $M$ along $K$.

If $p/q\neq0$,
then there is a canonical isomorphism $H^1(M_{p/q}(K);\Integral)\cong H^1(M;\Integral)$,
and accordingly,
the following equality of Alexander polynomials holds 
for any primitive cohomology class $\psi\in H^1(M;\Integral)$.
$$\Delta_{M_{p/q}(K)}^\psi(t)\doteq p\cdot\Delta_{M}^\psi(t)$$
Here, the dotted equal symbol means
an equality in $\Integral[t,t^{-1}]$ up to a monic factor and a sign.
\end{lemma}

\begin{proof}
	Take an embedded compact tubular neighborhood $D^2\times K$ of $K$ in $M$,
	with product parametrization consistent with a meridian-longitude pair.
	Denote $X=M\setminus \mathrm{int}(D^2\times K)$.
	For any primitive cohomology class $\psi\in H^1(M;\Integral)$, 
	and for any positive integer $m$,
	the $m$--cyclic cover $\tilde{M}^\psi_m$ of $M$ dual to $\psi$
	pulls back to an $m$--cyclic cover $\tilde{X}^\psi_m$ of $X$.
	
	Since $K$ is null-homologous, the tubular neighborhood $D^2\times K$
	has exactly $m$ lifts to $\tilde{M}^\psi_m$,
	the complement of their interiors is by definition $\tilde{X}^\psi_m$.	
	Since $K$ is totally rationally null-homologous,
	we can find some oriented compact surface $F$ bounded by $K$
	in a rationally null-homologous neighborhood $R$ of $K$ in $M$.
	Therefore, the $m$ lifts of $K$ to $\tilde{M}^\psi_m$
	are bounded by $m$ lifts of $F$ to $\tilde{M}^\psi_m$.
	It follows that 
	the $m$ lifts of the meridian of $K$ (on $\partial D^2\times K$)
	represent $\Integral$--linearly independent elements in 
	$H_1(\tilde{X}^\psi_m;\Integral)$,
	as is evident by intersection pairing
	with $H_2(\tilde{X}^\psi_m,\partial \tilde{X}^\psi_m;\Integral)$.
	
	Viewing $H_2(\tilde{X}^\psi_m;\Integral)$ as a $\Integral[\Integral/m\Integral]$--module
	with respect to the deck tranformation group $\Integral/m\Integral$,
	these elements span a free $\Integral[\Integral/m\Integral]$--submodule
	$L_m\cong\Integral[\Integral/m\Integral]$,
	so we obtain a natural isomorphism of $\Integral[\Integral/m\Integral]$--modules:
	$$H_1(M^\psi_m;\Integral)\cong H_1(\tilde{X}^\psi_m;\Integral)/L_m,$$
	by the van Kampen theorem and abelianization of fundamental groups.
	
	For any rational number $p/q$, denote for simplicity
	$$N=M_{p/q}(K).$$
	By simutaneously Dehn filling the lifted $p/q$--slopes on $\partial \tilde{X}^\psi_m$,
	we obtain an $m$--cyclic cover $\tilde{N}^{\psi}_m$ of $N$,
	and a natural isomorphism of $\Integral[\Integral/m\Integral]$--modules:
	$$H_1(N^\psi_m;\Integral)\cong H_1(\tilde{X}^\psi_m;\Integral)/p\cdot L_m.$$
	Therefore, we obtain a natural short exact sequence of $\Integral[\Integral/m\Integral]$--modules:
	$$\xymatrix{
	0 \ar[r] & L_m/p\cdot L_m \ar[r] & H_1(\tilde{N}^\psi_m;\Integral) \ar[r] 
	& H_1(\tilde{M}^\psi_m;\Integral) \ar[r] & 0
	}.$$
	In particular, for $p\neq0$, 
	this implies a canonical isomorphism
	$$H^1(N;\Integral)\cong H^1(M;\Integral),$$
	applying with $m=1$.
	With respect to this canonical isomorphism,
	$N^\psi_m$ is the $m$--cyclic cover of $N$ dual to $\psi$.
	
	Note that the similar argument also works for the infinite cyclic cover of $M$ dual to $\psi$,
	namely, $m=\infty$.
	This yields 
	a natural short exact sequence of modules over $\Integral[\Integral]\cong\Integral[t,t^{-1}]$:
	$$\xymatrix{
	0 \ar[r] & L_\infty/p\cdot L_\infty \ar[r] & H_1(\tilde{N}^\psi_\infty;\Integral) \ar[r] 
	& H_1(\tilde{M}^\psi_\infty;\Integral) \ar[r] & 0
	}.$$
	
	Suppose $p/q\neq0$.
	Then the $\Integral[t,t^{-1}]$--module 
	$L_\infty/p\cdot L_\infty\cong \Integral/p\Integral\otimes_\Integral \Integral[t,t^{-1}]$ 
	is obviously torsion with order $p$, up to a unit of $\Integral[t,t^{-1}]$.
	By the above short exact sequence over  $\Integral[t,t^{-1}]$,
	$H_1(M^\psi_\infty;\Integral)$ is torsion
	if and only if $H_1(\tilde{N}^\psi_\infty;\Integral)$ is torsion.
	Moreover, in the torsion case,
	the asserted equality of Alexander polynomials
	$$\Delta_{N}^\psi(t)\doteq p\cdot\Delta_{M}^\psi(t)$$
	holds in $\Integral[t,t^{-1}]$ up to a unit of $\Integral[t,t^{-1}]$.
	In the nontorsion case,
	the asserted equality also holds,
	as $\Delta_{N}^\psi(t)$ and $\Delta_{M}^\psi(t)$ both vanish.
\end{proof}

\begin{remark}\label{surgery_Alexander_remark}
	For any oriented connected closed irreducible $3$--manifold $M$
	and any oriented null-homologous knot $K\subset M$,
	there is a canonical isomorphism
	$H^1(M_{0}(K);\Integral)\cong H^1(M\setminus K;\Integral)$.
	Accordingly,
	for any primitive cohomology class $\psi\in H^1(M\setminus K;\Integral)$,
	one can figure out the following formula
	$$\Delta_{M\setminus K}^\psi(t)\cdot\left(t^{\psi(\mu)}-1\right)
	\doteq \Delta^\psi_{M_0(K)}(t)\cdot
	\begin{cases}
	(t-1) &b_1(M\setminus K)\geq2\\
	(t-1)^2 &b_1(M\setminus K)=1
	\end{cases}$$
	where $\mu\in H_1(M\setminus K;\Integral)$ is represented by the meridian of $K$ in $M$.
	See \cite[Theorem F]{Turaev_RA}, \cite[Section 3.3]{Friedl--Vidussi_survey}, 
	and \cite[Proposition 2.3]{Sun_vhsr} for facts to derive the formula.	
	
	This formula determines $\Delta^\psi_{M_0(K)}(t)$ unless $\psi(\mu)=0$.
	If $\psi(\mu)=0$,
	and if $K$ is totally rationally null-homologous in $M$,
	then $\psi$ lies in the image of 
	the natural homomorphism $H^1(M;\Integral)\to H^1(M\setminus K;\Integral)$.
	In this case, the finite cyclic covers of $M_0(K)$ dual to $\psi$
	have unbounded first Betti number, 
	implying $\Delta^\psi_{M_0(K)}(t)\doteq0$ by Lemma \ref{cyclic_Betti}.
\end{remark}

\section{Examples}\label{Sec-examples}
In this section,
we prove Theorems \ref{main_example_full} and \ref{main_example_partial}
by exhibiting Examples \ref{example_full} and \ref{example_partial},
respectively.
We also prove a simple observation as mentioned in the introduction,
namely, 
any counter example to the vanishing Alexander polynomial conjecture (Conjecture \ref{vanishing_AP})
leads to an example of the virtual Euler class one conjecture (Conjecture \ref{virtual_ecoc})
with first Betti number $\geq4$ (Proposition \ref{vanishing_AP_negative}).

\begin{example}\label{example_full}
	For any integer $e$,
	let $N=N(e)$ be an oriented circle bundle over an oriented torus with Euler number $e$.
	We observe 
	$$b_1(N)=
	\begin{cases}
	2 & e\neq0 \\
	3 & e=0
	\end{cases}
	$$
	Note that any primitive cohomology class $\psi\in H^1(N;\Integral)$ is fibered.
	With respect to $\psi$, the dual surface fiber of $N$ is an oriented torus.
	The forward monodromy is the identity, if $e=0$;
	or it is the (right-hand) Dehn twist raised to the power $e$,
	along the slope represented by the circle-bundle circle, if $e\neq0$.
	Therefore, for any primitive cohomology class $\psi\in H^1(N;\Integral)$,
	we observe
	$$\Delta_N^\psi(t)\doteq(t-1)^2$$
	as an equality in $\Integral[t,t^{-1}]$ up to a monomial factor and a sign.
	
	Take some oriented, null-homologous, totally rationally null-homologous, and hyperbolic knot $K\subset N$,
	and some hyperbolic surgery of $N$ along $K$ with rational coefficient $p/q\neq0$, 
	as guaranteed to exist by Lemma \ref{hyperbolic_totally_rnh}.
	Denote the resulting oriented closed hyperbolic $3$--manifold as $M=N_{p/q}(K)$.
	
	By Lemma \ref{surgery_Alexander},
	we obtain $b_1(M)=b_1(N)$, equal to either $2$ or $3$,
	and $\Delta^\psi_M(t)\doteq p\cdot\Delta^\psi_N(t)$,
	nonvanishing for any primitive cohomology class $\psi\in H^1(M;\Integral)$.
	By Theorem \ref{main_Alexander} and Corollary \ref{main_Alexander_corollary},
	every rational point $w$ on $\partial\mathcal{B}_{\mathrm{Th}^*}(M)$
	is virtually realizable by a taut foliation.
	Indeed, the pullback $\tilde{w}$ of $w$ to some cyclic finite cover $\tilde{M}$ of $M$
	dual to some primitive cohomology class in $H^1(M;\Integral)$
	is the real Euler class of some transversely oriented taut foliation on $\tilde{M}$.
\end{example}

\begin{example}\label{example_partial}
	For any integers $g\geq2$ and $0\leq s \leq g$,
	let $N=N(g,s)$ be an oriented surface bundle over an oriented circle,
	such that the surface fiber $S$ is connected closed of genus $g$,
	and the forward monodromy $f\colon S\to S$ 
	is the product of the Dehn twists
	along $s$ mutually disjoint simple closed curves
	of nonseparate union.
	We observe 
	$$b_1(N)=1+2g-s.$$
	In fact, $H^1(N;\Integral)$	splits 
	as the direct sum of fixed $\Integral$--submodule 
	$\mathrm{Fix}(f^*)\cong\Integral^{2g-s}$
	of the induced homomorphism $f^*\colon H^1(S;\Integral)\to H^1(S;\Integral)$
	and $\Integral\cdot\mathrm{PD}([S])\cong\Integral$.
	Moreover, the Thurston norm $\|\cdot\|_{\mathrm{Th}}$
	of $N$ vanishes on $\mathrm{PD}(\mathrm{Fix}(f^*))$
	and satisfies $\|[S]\|_{\mathrm{Th}}=2g-2$.
	Treating $N$ as the mapping torus of $f\colon S\to S$,
	the suspension flow through any point on $S$ outside the Dehn-twist annuli
	yields a $1$--periodic trajectory,
	representing a homology class $\Gamma\in H_1(N;\Integral)$.
	From the above description of $\|\cdot\|_{\mathrm{Th}}$,
	the dual Thurston norm unit ball $\mathcal{B}_{\mathrm{Th}^*}(N)$
	is contained in the linear subspace $\Real\cdot\Gamma$ of $H_1(N;\Real)$,
	as the closed interval with endpoints $\pm(2g-2)\cdot\Gamma$.
	
	The Alexander polynomial of $N$ 
	with respect to the distinguished fibered class 
	$\phi=\mathrm{PD}([S])$ in $H^1(N;\Integral)$
	can be computed as the characteristic polynomial of $f_*$,
	yielding $\Delta^\phi_N(t)\doteq(t-1)^{2g}$.
	The Newton polytope of $\Delta^\#_N$ has to be a segment
	due to the shape of $\mathcal{B}_{\mathrm{Th}^*}(N)$,
	so we infer $\Delta^\#_N\doteq(\Gamma-1)^{2g-2}$
	in $\Integral[H]$ up to a unit (see Theorem \ref{MAP_properties}).
	For any primitive cohomology class $\psi\in H^1(N;\Integral)$, we obtain
	$$\Delta^\psi_N(t)\doteq \left(t^{\langle\psi,\Gamma\rangle}-1\right)^{2g-2}\cdot(t-1)^2.$$
	
	Take $M=N_{p/q}(K)$ to be a hyperbolic surgery of $N$ with rational coefficient $p/q\neq0$,
	along an oriented, null-homologous, totally rationally null-homologous, and hyperbolic knot $K\subset N$
	(Lemma \ref{hyperbolic_totally_rnh}).
	Note that $b_1(M)=b_1(N)$ (Lemma \ref{surgery_Alexander})
	can be any integer $b=1+2g-s\geq3$,	by suitably setting $g$ and $s$.
	
	For any primitive cohomology class $\psi$
	in $H^1(M;\Integral)\cong H^1(N;\Integral)$,
	we obtain $\Delta^\psi_M(t)\doteq p\cdot\Delta^\psi_N(t)$ (Lemma \ref{surgery_Alexander}),
	vanishing if and only if $\psi$ lies in $\mathrm{Fix}(f^*)$.
	Since $M$ is hyperbolic,
	$\partial\mathcal{B}_{\mathrm{Th}}(M)$
	must have some pair of opposite vertices lying outside
	the codimension--$1$ linear subspace
	$\mathrm{PD}(\mathrm{Fix}(f^*))\otimes_\Integral\Real$ of $H_2(M;\Real)$.
	These vertices are dual to an opposite pair of top-dimensional closed faces
	$\pm F\subset\partial\mathcal{B}_{\mathrm{Th}^*}(M)$.
	Theorem \ref{main_Alexander} and Corollary \ref{main_Alexander_corollary}
	apply to $\pm F$,
	since we can use a primitive cohomology class $\psi\in H^1(M;\Integral)$
	with $\mathrm{PD}(\psi)$ an integral multiple of the above vertices.
	Therefore,
	every rational point $w\in\pm F$
	is virtually realizable by a taut foliation.
\end{example}

\begin{proposition}\label{vanishing_AP_negative}
	If there exists a counter example to Conjecture \ref{vanishing_AP} for some $b\geq4$,
	then there exists some oriented closed hyperbolic $3$--manifold $M$ with $b_1(M)=b$,
	such that every rational point $w\in\partial\mathcal{B}_{\mathrm{Th}^*}(M)$ 
	is virtually realizable by a taut foliation, and indeed,
	by similar means as specified in Theorem \ref{main_example_full}.
\end{proposition}

\begin{proof}
	Suppose that $N$ is a counter example to Conjecture \ref{vanishing_AP} for some integer $b\geq4$.
	This means that 
	$N$ is an oriented connected closed $3$--manifold with $b_1(N)=b$,
	such that for any primitive cohomology class $\psi\in H^1(N;\Integral)$,
	the Alexander polynomial $\Delta^\phi_N(t)$ does not vanish.
	Under this hypothesis,
	$N$ can take the place of the synonymous manifold in Example \ref{example_full},
	and the similar argument remains valid.
	Namely, 
	we apply Lemma \ref{hyperbolic_totally_rnh}
	to obtain a hyperbolic surgery of $N$ with rational coefficient $p/q\neq0$,
	along an oriented, null-homologous, totally rationally null-homologous, and hyperbolic knot $K\subset N$.
	The resulting oriented closed hyperbolic $3$--manifold
	$M=N_{p/q}(K)$ with $b_1(M)=b$
	satisfies the asserted property in Proposition \ref{vanishing_AP_negative},
	by Lemma \ref{surgery_Alexander}, and Theorem \ref{main_Alexander}, and Corollary \ref{main_Alexander_corollary}.
\end{proof}

\bibliographystyle{amsalpha}

\end{document}